\documentclass[12pt,notitlepage]{amsart}
\usepackage{latexsym,amsfonts,amssymb,amsmath,amsthm}
\usepackage{graphicx}
\usepackage{color}
\usepackage{multirow}
\usepackage[normalem]{ulem}
\usepackage{datetime2}
\usepackage{accents}
\usepackage{tikz}
\usetikzlibrary{positioning, arrows.meta}
\let\turc\c
\usepackage{etoolbox}
\robustify\turc %new thing to get the c in Topacogullari
\renewcommand{\c}{\mathfrak{c}}
\newcommand{\bpm}{\begin{pmatrix}}
\newcommand{\epm}{\end{pmatrix}}

\usepackage[inner=1.0in,outer=1.0in,bottom=1.0in, top=1.0in]{geometry}

\newcommand{\mz}{\ensuremath{\mathbb Z}}

\newcommand{\mymod}{\ensuremath{\negthickspace \negmedspace \pmod}}
\newcommand{\shortmod}{\ensuremath{\negthickspace \negthickspace \negthickspace \pmod}}

\newcommand{\intR}{\int_{-\infty}^{\infty}}

\newcommand{\sumstar}{\sideset{}{^*}\sum}

\theoremstyle{plain}		
	\newtheorem{mytheo}{Theorem} [section]
	
	\newtheorem{myprop}[mytheo]{Proposition}
	\newtheorem{mycoro}[mytheo]{Corollary}
     \newtheorem{mylemma}[mytheo]{Lemma}

\theoremstyle{remark}

\numberwithin{equation}{section}
\numberwithin{figure}{section}

\begin{document}
\title{Remarks on the distribution of  Dirichlet $L$-functions along cosets
}

 \author{Matthew P. Young}
 
 \address{Department of Mathematics \\
 	  Rutgers University \\
 	 Piscataway \\
 	  NJ 08854-8019 \\
 		U.S.A.}		
 \email{mpy4@rutgers.edu}
 
 \thanks{
This material is based upon work supported by the National Science Foundation under agreement No.
DMS-2302210. Any opinions, findings and conclusions or recommendations expressed in this material
are those of the authors and do not necessarily reflect the views of the National Science Foundation. 
 } 

 \begin{abstract}
 In a previous work with B. Garcia, the author considered the asymptotic for the second moment of Dirichlet $L$-functions along cosets, and exhibited a surprising secondary main term that is not predicted by the recipe of Conrey, Farmer, Keating, Rubinstein, and Snaith.  In this paper, we re-examine this problem and propose a modified recipe that correctly predicts this secondary main term.  The original recipe gives the incorrect answer for this family because the root number is not always independent of the Dirichlet series coefficients along certain cosets, and our proposed fix simply takes this feature into account.

In addition, we consider a handful of other problems related to Dirichlet $L$-functions along cosets.  One goal is to reformulate Heath-Brown's $q$-analog of van der Corput's shifting method in terms of cosets, which leads to an upper bound on a hybrid second moment.  We also revisit the classical van der Corput bound and view it (in more modern terms) as an amplified second moment of a trigonometric polynomial.  
 \end{abstract}

\dedicatory{To Roger Heath-Brown, on the occasion of his 75th birthday}
 \maketitle
 \section{Introduction}
The aim of this article is to explore the behavior of Dirichlet $L$-functions along cosets.
To start the discussion, we present the following result which in hindsight is largely implicit in the work of Heath-Brown \cite{HeathBrownHybrid}.
\begin{mytheo}
\label{thm:HBmoment}
Let $q$ be odd, $q_0 | q$, $1 \leq T_0 \leq T$,   and let $\chi$ be a Dirichlet character of conductor $q$.  Then
\begin{equation}
\label{eq:HBmoment}
\int_{T}^{T+T_0} \sum_{\psi \shortmod{q_0}} 
|L(1/2+it , \chi \cdot \psi)|^2 dt \ll 
(qT)^{\varepsilon}
(T_0+T_0^{-1/2} T^{1/2})
(q_0  + q_0^{-1/2} q^{1/2}).
\end{equation}
\end{mytheo}
A remarkable feature of this bound is the analogous roles played by $T_0$ and $q_0$ (resp. $T$ and $q$).  As Heath-Brown observes, the resulting bound is optimized if $q$ has a divisor $q_0$ with $q_0 \approx q^{1/3}$, in which case one may deduce
$L(1/2+it, \chi) \ll (qT)^{1/6+\varepsilon}$, which is the hybrid Weyl bound.
For emphasis, we note that the sum over $\psi$ means that $\chi \cdot \psi$ runs over a coset of characters, namely the set $\chi \cdot H_{q_0}$, where $H_{q_0}$ is the group of Dirichlet characters modulo $q_0$, which may be viewed as a subgroup of $H_{q}$.  It has been argued elsewhere (e.g., \cite{PetrowYoungCoset}) that the continuous short interval $[T, T+T_0]$ acts as an archimedean analog to the discrete coset $\chi H_{q_0}$, in the following sense.  Suppose that $t_1, t_2 \in [T, T+T_0]$.  Then $t_1 - t_2 \in [-T_0, T_0]$, which means that the conductor of the character $n \mapsto n^{it_1 - it_2}$ is $\ll T_0$.  Similarly, if $\eta_1, \eta_2 \in \chi H_{q_0}$, then $\eta_1 \overline{\eta_2} \in H_{q_0}$, so the conductor of $\eta_1 \overline{\eta_2}$ divides $q_0$.  

Heath-Brown's original proof did not directly study a second moment along a coset.  Rather, he introduced a $q$-analog of the van der Corput shifting method, from which he deduced an upper bound on an individual character sum.  

One may find variants of the second moment bound \eqref{eq:HBmoment} in \cite{MW, Nunes, GarciaYoung}, which have different hypotheses and restrict to the central value.  

We sketch a proof of Theorem \ref{thm:HBmoment} in Section \ref{section:upperbound} below.  
Along the way, we also take the opportunity to discuss some features of the van der Corput shifting method that could be of independent interest.  In particular, this discussion should help clarify why Heath-Brown's $q$-analog can be interpreted as a coset average (one can also find some similar observations in \cite[\S 3.1]{MW}, for instance).  
We will additionally express the classic van der Corput shifting method in terms of an amplified second moment, which seems to be new.

As a supplement to Theorem \ref{thm:HBmoment}, one may also consider the problem of an asymptotic formula for the second moment.  Some results in this vein were given in \cite{GarciaYoung}.  For simplicity of presentation, we will focus on the special cases where the conductor $q$ is an odd prime power.
To state the result, we need to define the following parameter.  
For $k \geq 2$, and $p \neq 2$, we let $\ell_{\chi} \in \mz/p^{k-1} \mz$ be defined by
\begin{equation*}
\chi(1+px) = e_{p^k}(\ell_{\chi} \log(1+px)),
\end{equation*}
where $\log$ here is the $p$-adic logarithm which may be defined by the usual power series expansion.  We will only have need of the linear and quadratic approximations, specifically,
\begin{equation*}
\log(1+px) 
\equiv 
\begin{cases}
px \pmod{p^k}, \qquad &\text{if } (px)^2 \equiv 0 \pmod{p^k}, \\
px - 2^{-1} (px)^2 \pmod{p^k}, \qquad & \text{if } 3^{-1} (px)^3 \equiv 0 \pmod{p^k}.
\end{cases}
\end{equation*}
We also define $H_{p^j}^{\pm} = \{ \psi \in H_{p^j} : \psi(-1) = \pm 1 \}$.
\begin{mytheo}[\cite{GarciaYoung}, Theorem 1.1]
\label{thm:GYsecondmoment}
Let $\chi$ be even of conductor $q=p^k$ and suppose $q_0 = p^j$, with $j < k \leq 2j$.  
Define $a_{\chi}$ by $a_{\chi} \equiv \ell_{\chi} \pmod{q/q_0}$, with $|a_{\chi}| < \frac{q}{2q_0}$.
Then
\begin{equation}
\sum_{\psi \in H_{q_0}^{+}} |L(1/2, \chi \cdot \psi)|^2
= \mathcal{D} + \mathcal{A} + O(q^{-1/8+\varepsilon} q_0),
\end{equation}
where with $\gamma=0.577\dots$ the Euler constant,
\begin{equation}
\label{eq:diagonalTermDef}
\mathcal{D} = \frac{\varphi(q_0)}{2} \frac{\varphi(q)}{q} (\log q + 2 \gamma + \tfrac{\Gamma'}{\Gamma}(1/4) - \log(\pi) + 2 \tfrac{\log{p}}{p-1}),
\end{equation}
and where
\begin{equation}
\mathcal{A} = \frac{\varphi(q_0)}{q_0} \sqrt{q} \frac{d(|a_{\chi}|)}{\sqrt{|a_{\chi}|}}.
\end{equation}
\end{mytheo} 
Here $\mathcal{D}$ is the easily-predicted diagonal main term, but the presence of $\mathcal{A}$ is rather surprising.  In \cite{GarciaYoung}, it was argued that the term $\mathcal{A}$ goes against the recipe of \cite{CFKRS}.
However, the issue is a bit more nuanced. 
The heart of the matter is that the behavior of the root number along the coset $\chi H_{q_0}$ is highly irregular.  For instance, there exists a choice of $\chi$ such that for $k \geq 2$ even, $\chi H_{p^{k/2}}^{+}$ has constant root number $\epsilon(\psi)$ which is very close to $1$ (see Corollary \ref{coro:rootnumbercloseto1} below for a precise statement).  More generally, for each $m$ coprime to $p$ there exists a coset such that $\epsilon(\psi) \cdot \overline{\psi}(m)$ is constant.  If $m \neq 1$, then the root number therefore averages to zero (along the coset), by the usual orthogonality properties enjoyed by Dirichlet characters.  However, this correlation with a character value $\psi(m)$ is enough to distort the behavior of the central values.

To address how this relates to the recipe of \cite{CFKRS}, let us follow through the recipe for the second moment which appears in Theorem \ref{thm:GYsecondmoment}.  We introduce two shift parameters and consider the formal quantity
\begin{equation}
\label{eq:formalquantity}
\begin{split}
\sum_{\psi \in H_{q_0}^{+}} 
&\Big( \sum_{n} \frac{\chi(n) \psi(n)}{n^{1/2+\alpha}} 
+ X_{\alpha} \cdot  \epsilon(\chi \cdot \psi) \sum_{n} \frac{\overline{\chi(n) \psi(n)}}{n^{1/2-\alpha}} \Big)
\\
\cdot
&\Big( \sum_{n} \frac{\overline{\chi(n) \psi(n)}}{n^{1/2+\beta}} 
+ X_{\beta} \cdot \epsilon(\overline{\chi \cdot \psi}) \sum_{n} \frac{\chi(n) \psi(n)}{n^{1/2-\beta}} \Big),
\end{split}
\end{equation}
where $X_{\alpha}$ and $X_{\beta}$ are terms arising from the gamma factors in the functional equation.  The next step in the recipe tasks us to multiply out these sums and ``replace each product of $\epsilon$-factors by its expected value when averaged over the family" (see \cite[p.80, step 3]{CFKRS}).  In the case that $\chi H_{p^{k/2}}$ has constant root number, then it turns out that the recipe does lead to an answer consistent with Theorem \ref{thm:GYsecondmoment}, and this was overlooked in the discussions in \cite{GarciaYoung} (in that paper, Garcia and the present author implicitly assumed that the root number averages to zero). 
 However, in a case where the coset $\chi H_{p^{k/2}}$ has root number that is correlated with a character value $\psi(m)$ (with $m \neq 1$), then the recipe would tell us to discard it.  We propose to correct this procedure by instead multiplying out the terms and replacing each product of $\epsilon$-factors \emph{times Dirichlet series coefficients} with its expected value when averaged over the family.  Since step (4) of the recipe is to replace each summand by its expected value when averaged over the family, we are proposing here to combine these two steps into one rather than to keep them separate.  However, for many families the root number is essentially independent from the Dirichlet series coefficients, in which case our proposed modification of the recipe is equivalent to the original form of \cite{CFKRS}.

In Section \ref{section:recipe} below we show that this procedure does indeed produce an asymptotic formula consistent with Theorem \ref{thm:GYsecondmoment}.  It would be interesting to find further examples of families where the root number is not effectively independent from the Dirichlet series coefficients.

The fact that there are multiple characters of prime power conductor with the same Gauss sum was apparently first observed in \cite{KMP, Molteni}, though in retrospect it is implicit in \cite{Odoni}.  

This article largely focuses on the behavior of $L$-functions along cosets of the form $\chi H_{p^j}$ inside $H_{p^k}$, with $1 \leq j < k$.  
There are various works 
which consider a complementary situation of cosets of subgroups of $H_p$, including \cite{KMN, MS, DKMS}.

\section{Background on characters and Gauss sums}
In this section we collect some material on Dirichlet characters.
Let $\chi$ be a Dirichlet character of conductor $p^k$, where $p$ is an odd prime and $k \geq 1$.  Let $\tau(\chi)$ be the standard Gauss sum defined as $\tau(\chi) = \sum_{t \mymod{p^k}} \chi(t) e_{p^k}(t)$, where $e_{q}(x) = e(x/q)$ and $e(x) = \exp(2 \pi i x)$.  It is well-known that $|\tau(\chi)| = p^{k/2}$.  More generally, for $n \in \mz$, let
\begin{equation*}
\tau(\chi, n) = \sum_{t \shortmod{p^k}} \chi(t) e_{p^k}(nt) = \overline{\chi}(n) \tau(\chi).
\end{equation*}
In addition, let $\epsilon(\chi)$ be the root number of $L(s, \chi)$, and define
\begin{equation*}
\epsilon_{q} = 
\begin{cases}
1, \qquad q \equiv 1 \pmod{4}, \\
i, \qquad q \equiv 3 \pmod{4}.
\end{cases}
\end{equation*}
For $\chi$ even, $\epsilon(\chi) = q^{-1/2} \tau(\chi)$.

We have need of the standard evaluation of the quadratic exponential sum.
\begin{mylemma}
\label{lemma:quadraticGaussEval}
Suppose $q$ is an odd, positive integer, and $\gcd(a,q) = 1$.  Then
\begin{equation*}
\sum_{x \shortmod{q}} e_q(ax^2 + bx)
= q^{1/2} \epsilon_q \Big(\frac{a}{q}\Big) e_q(- \overline{4a} b^2).
\end{equation*}
\end{mylemma}
This can be proved by completing the square and using say \cite[(3.38)]{IK}.

We quote the following version of the Postnikov formula from \cite[Lemma 2.1]{PetrowYoungCoset}.
\begin{mylemma}[Postnikov formula]
\label{lemma:Postnikov}
The map $\chi \mapsto \ell_{\chi}$ is a surjective group homomorphism from $H_{p^k}$ to $\mz/{p^{k-1} \mz}$.  For $1 \leq j \leq k$ we have that $\ell_{\chi_1} \equiv \ell_{\chi_2} \pmod{p^{k-j}}$ if and only if the conductor of $\chi_1 \overline{\chi_2}$ divides $p^{j}$.
\end{mylemma}
A simple consequence of Lemma \ref{lemma:Postnikov} is that if $\chi_1, \chi_2 \in \chi H_{p^j} \subset H_{p^k}$, with $1 \leq j < k$, then $\ell_{\chi_1} \equiv \ell_{\chi_2} \pmod{p^{k-j}}$.

The following is essentially due to Odoni \cite{Odoni}.
\begin{mylemma}
Let $p$ be odd, $k \geq 2$, and suppose $\chi$ has conductor $p^k$.
If $k=2n$, then
\begin{equation}
\label{eq:tauEvalEven}
\tau(\chi) = p^{n} \sumstar_{\substack{t_0 \shortmod{p^{n}} \\ t_0 \equiv - \ell_{\chi} \shortmod{p^{n}}}}
\chi(t_0) e_{p^k}(t_0).
\end{equation}
If $k=2n+1$ and $p \geq 5$, then
\begin{equation}
\label{eq:tauEvalOdd}
\tau(\chi) = \epsilon_p p^{n+\frac12} \Big(\frac{-2 \ell_{\chi}}{p}\Big)
\sumstar_{\substack{t_0 \shortmod{p^n} \\ t_0 \equiv - \ell_{\chi} \shortmod{p^n}}} 
\chi(t_0) e_{p^k}(t_0)  
e_p\Big( \overline{2 \ell_{\chi}} \big(\frac{\ell_{\chi} + t_0}{p^n}\big)^2 \Big).
\end{equation}
In both \eqref{eq:tauEvalEven} and \eqref{eq:tauEvalOdd}, each summand is periodic in $t_0 \pmod{p^n}$, subject to the congruence $t_0 \equiv - \ell_{\chi} \pmod{p^n}$.
\end{mylemma}
Remark. In both \eqref{eq:tauEvalEven} and \eqref{eq:tauEvalOdd}, the sum collapses to a single term, but the given presentation is intended to emphasize its periodicity modulo $p^n$.
\begin{proof}
Let $k = 2n + \delta$ with $\delta \in \{0, 1 \}$.
In the definition of the Gauss sum, write $t = t_0(1+p^n t_1)$ where $t_0$ runs modulo $p^n$ (coprime to $p$) and $t_1$ runs modulo $p^{k-n}$.  
Here $t_0$ and $t_1$ each may run over any complete set of residues for their respective moduli.
Then 
\begin{equation*}
\chi(t) = \chi(t_0) e_{p^k}(\ell_{\chi} \log(1 + p^n t_1))
= \chi(t_0) e_{p^{k-n}}(\ell_{\chi} [t_1 - 2^{-1} p^n t_1^2]).
\end{equation*}
Therefore,
\begin{equation*}
\tau(\chi) = \sumstar_{t_0 \shortmod{p^n}} \chi(t_0) e_{p^k}(t_0)
\sum_{t_1 \shortmod{p^{k-n}}} e_{p^{k-n}}(\ell_{\chi} [t_1 - 2^{-1} p^n t_1^2] + t_0 t_1).
\end{equation*}
If $k$ is even then the quadratic term may be discarded, and the sum over $t_1$ vanishes unless $t_0 \equiv - \ell_{\chi} \pmod{p^n}$.  Hence it simplifies as given by \eqref{eq:tauEvalEven}.

If $k$ is odd, we let $t_1 = u + p v$ where $u$ runs modulo $p$ and $v$ runs modulo $p^n$ (again, each running over any complete set of residues).  The $v$-sum evaluates in a similar way to the $t_1$-sum in the case that $k$ is even, giving
\begin{equation*}
\tau(\chi) = p^n \sumstar_{\substack{t_0 \shortmod{p^n} \\ t_0 \equiv - \ell_{\chi} \shortmod{p^n}}} 
\chi(t_0) e_{p^k}(t_0) 
\sum_{u \shortmod{p}} e_{p^{n+1}}( (\ell_{\chi} + t_0) u) e_p(- 2^{-1} \ell_{\chi} u^2).
\end{equation*}
The inner sum over $u$ may be calculated by Lemma \ref{lemma:quadraticGaussEval}, giving
\eqref{eq:tauEvalOdd}.

Finally, we address the periodicity of the summands modulo $p^{n}$.  This follows from the derivation of the formula, since the $t_0$-sum was able to run over any set of integer representatives modulo $p^n$.  Changing $t_0$ modulo $p^n$ amounts to changing the representative, and so by construction, the summand can not depend on the choice of residue.  (Alternatively, this fact may be checked directly.)
\end{proof}

\begin{mycoro}
\label{coro:gausssum}
Suppose that $\chi_1$ and $\chi_2$ have conductor $p^k$ and $\gcd(m,p) = 1$.  If $\chi_1 \overline{\chi_2}$ has conductor dividing $p^{\lceil k/2 \rceil}$, then
\begin{equation}
\label{eq:gausssumRatioFormula}
\frac{\tau(\chi_1, m)}{\tau(\chi_2, m)} = (\chi_1 \overline{\chi_2})(-\ell_{\chi_1}/m).
\end{equation}
In particular:
\begin{itemize}
\item If $m \equiv - \ell_{\chi_1} \pmod{p^{\lceil k/2 \rceil}}$, then $\tau(\chi_1, m)$ is constant on the coset $\chi_1 H_{p^{\lceil k/2 \rceil}}$.
\item If $m \equiv  \ell_{\chi_1} \pmod{p^{\lceil k/2 \rceil}}$, then $\tau(\chi_1, m)$ is constant on the set $\chi_1 H_{p^{\lceil k/2 \rceil}}^{+}$.
\end{itemize}
\end{mycoro}
\begin{proof}
Say $k=2n$ is even, and that $\chi_1 \overline{\chi_2}$ has conductor dividing $p^{n}$.  By \eqref{eq:tauEvalEven},
\begin{equation*}
\frac{\tau(\chi_1, m)}{\tau(\chi_2, m)} = \frac{\overline{\chi_1}(m) \tau(\chi_1)}{\overline{\chi_2}(m) \tau(\chi_2)}
= 
\frac
{\overline{\chi_1}(m) p^n \sum_{t_0 \equiv - \ell_{\chi_1} \shortmod{p^n}} \chi_1(t_0) e_{p^k}(t_0)}
{\overline{\chi_2}(m) p^n \sum_{t_0 \equiv - \ell_{\chi_2} \shortmod{p^n}} \chi_2(t_0) e_{p^k}(t_0)}.
\end{equation*}
By Lemma \ref{lemma:Postnikov}, the condition that $\chi_1 \overline{\chi_2}$ has conductor dividing $p^n$ implies that $\ell_{\chi_1} \equiv \ell_{\chi_2} \pmod{p^{n}}$.  Hence we may use the substitution $t_0 \equiv - \ell_{\chi_1} \pmod{p^n}$ in both the numerator and denominator above.  The ratio of Gauss sums then simplifies as claimed in \eqref{eq:gausssumRatioFormula}.

Now say $k= 2n+1$, and $\chi_1 \overline{\chi_2}$ has conductor dividing $p^{n+1}$.  Lemma \ref{lemma:Postnikov} implies $\ell_{\chi_1} \equiv \ell_{\chi_2} \pmod{p^n}$.
The proof of \eqref{eq:gausssumRatioFormula} is the same as for $k$ even, but using \eqref{eq:tauEvalOdd} in place of \eqref{eq:tauEvalEven}.  

Next we explain the final statement of the corollary.  If $m \equiv - \ell_{\chi_1} \pmod{p^{\lceil k/2 \rceil}}$ then $-\ell_{\chi_1}/m \equiv 1 \pmod{p^{\lceil k/2 \rceil}}$.  Since $\chi_1 \overline{\chi_2}$ has conductor dividing $p^{\lceil k/2 \rceil}$, this implies \eqref{eq:gausssumRatioFormula} simplifies as $\chi_1 \overline{\chi_2}(1) = 1$.  Finally, note that the condition that $\chi_1 \overline{\chi_2}$ has conductor dividing $p^{\lceil k/2 \rceil}$ is equivalent to $\chi_2 \in \chi_1 H_{p^{\lceil k/2 \rceil}}$.  
Similarly, if $m \equiv \ell_{\chi_1} \pmod{p^{\lceil k/2 \rceil}}$ and $\chi_1(-1) = \chi_2(-1)$, then \eqref{eq:gausssumRatioFormula} simplifies as $1$.
\end{proof}

Another special case is the following.
\begin{mycoro}
\label{coro:rootnumbercloseto1}
Suppose $\chi$ has conductor $p^k$ with $k=2n$, and $\ell_{\chi} \equiv  - 1 \pmod{p^{n}}$.  Then for $\psi \in \chi H_{p^{n}}$,
\begin{equation*}
\tau(\psi) = p^n e_{p^k}(1) = p^n (1+ O(p^{-k})) .
\end{equation*}
\end{mycoro}
\begin{proof}
This follows from \eqref{eq:tauEvalEven}, with the substitution $t_0 = 1$.
\end{proof}

Corollary \ref{coro:gausssum} clearly indicates that the root number behaves erratically on cosets.  However, what we need in practice are more general summation formulas, which are provided with the following lemma.
\begin{mylemma}
\label{lemma:gausssumsaverageLinearApprox}
Suppose that $\chi$ is even of conductor $q = p^k$ with $k \geq 2$ and $p$ odd.  
Suppose $k/2 \leq j < k$
and $\gcd(m,p) = 1$.  Then
\begin{equation*}
\sum_{\eta \in \chi H_{p^j}^{+}} \epsilon(\eta) \overline{\eta}(m) = 
q^{1/2}
\frac{\varphi(p^j)}{2p^j}
\sum_{\pm}
e_{p^k}(\pm m)  \delta(\ell_{\chi} \equiv \mp m \mymod{p^{k-j}}).
\end{equation*}
\end{mylemma}
Remark: the initial work in the proof treats any $j \geq 1$, which will be helpful later.
\begin{proof}
Since $\eta$ is even, then $\epsilon(\eta) = q^{-1/2} \tau(\eta)$.  Hence
\begin{multline}
\label{eq:Sformula}
S: =\sum_{\eta \in \chi H_{p^j}^{+}} \epsilon(\eta) \overline{\eta}(m)
=  \sum_{\psi \in H_{p^j}} \overline{\chi \psi}(m) \frac{1 + \psi(-1)}{2 \sqrt{q}}
\sum_{t \shortmod{p^k}} \chi \psi(t) e_{p^k}(t)
\\
= 
\sum_{\pm} \frac{\varphi(p^j)}{2q^{1/2}} 
\sumstar_{\substack{t \shortmod{p^k} \\ t \equiv \pm m \shortmod{p^j}}} \overline{\chi}(m) \chi(t) e_{p^k}(t).
\end{multline}
Next we write $t = \pm m(1+p^j u)$, where $u$ runs modulo $p^{k-j}$.  Note that
$\overline{\chi}(m) \chi(\pm m(1+p^j u)) = \chi(1+p^j u) = e_{p^{k-j}}(\ell_{\chi} u)$, since $\chi(-1) = 1$ and $2j \geq k$.  Therefore,
\begin{align*}
S = 
\sum_{\pm} \frac{\varphi(p^j)}{2q^{1/2}} e_{p^k}(\pm m)
\sum_{\substack{u \shortmod{p^{k-j}}}} e_{p^{k-j}}(\ell_{\chi} u \pm m u).
\end{align*}
The sum over $u$ is evaluated by orthogonality of characters, giving the claimed formula.
\end{proof}

We can also extend the calculation from Lemma \ref{lemma:gausssumsaverageLinearApprox} as follows.
\begin{mylemma}
\label{lemma:gausssumsaverageQuadraticApprox}
Suppose that $\chi$ is even of conductor $q = p^k$ with $k \geq 2$ and $p\geq 5$.  
Suppose $k/3 \leq j \leq k/2$
and $\gcd(m,p) = 1$.  Then
\begin{equation}
\label{eq:gausssumsaverageQuadraticApprox}
\sum_{\eta \in \chi H_{p^j}^{+}} \epsilon(\eta) \overline{\eta}(m) = 
\varphi(p^j) \sum_{\pm} \tfrac12 e_{p^k}(\pm m) \delta(\ell_{\chi} \equiv \mp m \shortmod{p^j})
\Big(\frac{- 2 \ell_{\chi}}{q}\Big) \epsilon_q 
e_{p^{k-2j}}\Big( \overline{2 \ell_{\chi}} \big(\frac{\ell_{\chi} \pm m}{p^j}\big)^2 \Big).
\end{equation}
\end{mylemma}
\begin{proof}
We pick up from \eqref{eq:Sformula}, which is valid for any $j \geq 1$.  We write $t = \pm m(1+p^j u)$, with $u$ running modulo $p^{k-j}$.  Then
\begin{equation*}
\overline{\chi}(m) \chi(\pm m (1+ p^j u)) = \chi(1+p^j u)
= e_{p^{k-j}}( \ell_{\chi} u) e_{p^{k-2j}}(-2^{-1} \ell_{\chi} u^2),
\end{equation*}
so that
\begin{equation*}
S = 
\sum_{\pm} \frac{\varphi(p^j)}{2q^{1/2}} e_{p^k}(\pm m)
\sum_{\substack{u \shortmod{p^{k-j}}}} e_{p^{k-j}}(\ell_{\chi} u \pm m u) e_{p^{k-2j}}(-2^{-1} \ell_{\chi} u^2).
\end{equation*}
Now write $u = u_1 + p^{k-2j} u_2$ where $u_1$ runs modulo $p^{k-2j}$ and $u_2$ runs modulo $p^{j}$.  Then $S$ equals
\begin{equation*}
\sum_{\pm} \frac{\varphi(p^j)}{2q^{1/2}} e_{p^k}(\pm m)
\sum_{\substack{u_1 \shortmod{p^{k-2j}}}} e_{p^{k-j}}(\ell_{\chi} u_1 \pm m u_1 -2^{-1} p^j \ell_{\chi} u_1^2)
\sum_{u_2 \shortmod{p^j}} e_{p^{j}}(\ell_{\chi} u_2 \pm m u_2).
\end{equation*}
The sum over $u_2$ evaluates by orthogonality of characters, giving
\begin{equation*}
S = \sum_{\pm} \frac{\varphi(p^j) p^j}{2q^{1/2}} e_{p^k}(\pm m) \delta(\ell_{\chi} \equiv \mp m \mymod{p^j})
\sum_{\substack{u_1 \shortmod{p^{k-2j}}}} e_{p^{k-2j}}\Big( \frac{\ell_{\chi}  \pm m}{p^j} u_1 -2^{-1} \ell_{\chi} u_1^2\Big).
\end{equation*}
Now the quadratic exponential sum evaluates in closed form by Lemma \ref{lemma:quadraticGaussEval}. 
Since $k$ and $k-2j$ have the same parity, this means $\epsilon_{p^k} = \epsilon_{p^{k-2j}}$ and $(\frac{-2 \ell_{\chi}}{p^{k-2j}}) = (\frac{-2 \ell_{\chi}}{p^k})$. 
Altogether, this gives the claimed formula.
\end{proof}

\section{The modified recipe}
\label{section:recipe}
\begin{myprop}
\label{prop:recipe1}
Let conditions be as in Theorem \ref{thm:GYsecondmoment}.
Then the modified recipe described in the introduction predicts the true asymptotic formula consistent with Theorem \ref{thm:GYsecondmoment}.
\end{myprop}

\begin{proof}
We continue with \eqref{eq:formalquantity}.  After distributing out the four sums, we encounter: two without any root number, one with $\epsilon$, and one with $\overline{\epsilon}$.  Consider the first of two terms without the root number, namely
\begin{equation*}
\sum_{\psi \in H_{q_0}^{+}} 
\sum_{n} \frac{\chi(n) \psi(n)}{n^{1/2+\alpha}} 
 \sum_{m} \frac{\overline{\chi(m) \psi(m)}}{m^{1/2+\beta}} 
 =
 \frac{\varphi(q_0)}{2} \sum_{\pm} \sum_{m \equiv \pm n \shortmod{q_0}} \frac{\chi(n) \overline{\chi(m)}}{n^{1/2+\alpha} m^{1/2+\beta}}.
\end{equation*}
The heuristic of the recipe is that the congruence $m \equiv \pm n \pmod{q_0}$ should be replaced by $m = n$ (see \cite[p.73]{CFKRS}).  Hence we obtain a diagonal term of the form
\begin{equation*}
\frac{\varphi(q_0)}{2} L(1 +\alpha + \beta, \chi_{q_0}).
\end{equation*}
The other term without the root number is of the same form but with $\alpha$ and $\beta$ swapped with their negatives, and multiplied by $X_{\alpha} X_{\beta}$.  Combining them gives
\begin{equation*}
\frac{\varphi(q_0)}{2} \Big(L(1 +\alpha + \beta, \chi_{q_0}) + X_{\alpha} X_{\beta} L(1-\alpha - \beta, \chi_{q_0}) \Big).
\end{equation*}
One can then let $\alpha=\beta=0$, since the singularity at this point is removable.
This is the same shape as the main term in the second moment of Dirichlet $L$-functions of conductor $q_0$, which is well-known to agree with $\mathcal{D}$.  For instance, this can be derived from \cite{HBsecondmoment} (in his notation, it is of the form $\frac{\phi(q_0)}{q_0} (T(q_0) - T(q_0/p))$).

Next we turn to the cross term with $\epsilon$, namely
\begin{equation}
\label{eq:epsiloncrossterm}
X_{\alpha} \sum_{\psi \in H_{q_0}^{+}} 
\epsilon(\chi \cdot \psi)
\sum_{n} \frac{\overline{\chi(n) \psi(n)}}{n^{1/2-\alpha}} 
\sum_{m} 
\frac{\overline{\chi(m) \psi(m)}}{m^{1/2+\beta}}.
\end{equation}
By Lemma \ref{lemma:gausssumsaverageLinearApprox}, the sum over $\psi$ evaluates \eqref{eq:epsiloncrossterm} as
\begin{equation*}
X_{\alpha} 
q^{1/2}
\frac{\varphi(q_0)}{2q_0}
\sum_{\pm}
\sum_{m,n} \frac{e_{p^k}(\pm mn)}{n^{1/2-\alpha} m^{1/2+\beta}} 
  \delta(\ell_{\chi} \equiv \mp mn \mymod{q/q_0}).
\end{equation*}
Recall that $a_{\chi}$ is defined by $a_{\chi} \equiv \ell_{\chi} \pmod{q/q_0}$ and $|a_{\chi}| < \frac{q}{2q_0}$.  The idea is that in the sum over $m$ and $n$ we only need to keep the term $mn = \mp a_{\chi}$.  Moreover, we have that $e_{p^k}(-a_{\chi}) = e(O(1/q_0)) = 1 + O(q_0^{-1})$, and this error term should be negligible.  Hence this cross term should take the form
\begin{equation*}
X_{\alpha} q^{1/2} \frac{\varphi(q_0)}{2q_0} \sum_{mn=|a_{\chi}|} \frac{1}{n^{1/2-\alpha} m^{1/2+\beta}}.
\end{equation*}
Unlike the diagonal term, this term is analytic for $\alpha, \beta$ in a neighborhood of zero, so we can substitute $\alpha = \beta = 0$ to simplify this term as $\tfrac12 q^{1/2} \frac{\varphi(q_0)}{q_0} \frac{d(|a_{\chi}|)}{\sqrt{|a_{\chi}|}}$, which is precisely $\tfrac12 \mathcal{A}$.  The other cross term will give the same contribution, and hence this gives rise to $\mathcal{A}$.
\end{proof}
Remark.  The proof of Theorem \ref{thm:GYsecondmoment} in \cite{GarciaYoung} used a two-piece approximate functional equation for $L(1/2, \chi \psi) L(1/2, \overline{\chi \psi})$ and as such the root number did not appear explicitly in the analysis there.  Instead, the main term arose only after Poisson summation in a range where the two variables were ``unbalanced."  Some heuristics for how the main term in this setup can be seen in \cite{GarciaYoung}, equation (1.9) and surrounding discussion.

Garcia and the author also proved a variant on Theorem \ref{thm:GYsecondmoment} but for somewhat smaller values of $j$.  The secondary main term takes a curiously different shape.  To state their result we need the following notation.  For $p \geq 5$ and $2j \leq k \leq 3j$, define $a_{\chi} \in \mz$ by $a_{\chi} \equiv \ell_{\chi} \pmod{p^{k-j}}$ with $|a_{\chi}| < \tfrac12 p^{k-j}$ (consistent with the previous definition of $a_{\chi}$) and further define $b_{\chi} \in \mz$ by $b_{\chi} \equiv a_{\chi} \pmod{p^j}$, with $|b_{\chi}| < \tfrac12 p^j$.  We note in passing that if $j=k/2$ then $b_{\chi} = a_{\chi}$ but for $j<k/2$ then it is certainly possible that $|b_{\chi}| < |a_{\chi}|$.
\begin{mytheo}[\cite{GarciaYoung}, Theorem 1.2.]
\label{thm:GYsecondmoment2}
Let $\chi$ be even of conductor $q=p^k$ with $p \geq 5$, and suppose $q_0 = p^j$, with $2j \leq k \leq 3j$.  Then
\begin{equation}
\sum_{\psi \in H_{q_0}^{+}} |L(1/2, \chi \cdot \psi)|^2
= \mathcal{D} + \mathcal{A}' + O(q_0^{-1/4} q^{1/2+\varepsilon}),
\end{equation}
where $\mathcal{D}$ is as defined in \eqref{eq:diagonalTermDef}, and with $a_{\chi} \overline{a_{\chi}} \equiv 1 \pmod{p^k}$, we have
\begin{equation}
\label{eq:A'def}
\mathcal{A}' = 
\Big(\frac{2 a_{\chi}}{q}\Big) \cdot \varphi(q_0) \frac{d(|b_{\chi}|)}{\sqrt{|b_{\chi}|}}
\times
\begin{cases}
\cos\Big(2 \pi \frac{\overline{2 a_{\chi}} (a_{\chi} - b_{\chi})^2}{q}\Big), \quad &q \equiv 1 \pmod{4}, \\
\sin\Big(2 \pi \frac{\overline{2 a_{\chi}} (a_{\chi} - b_{\chi})^2}{q}\Big), \quad &q \equiv 3 \pmod{4}.
\end{cases}
\end{equation}
\end{mytheo}
Remark.  It is worth mentioning that when $k=2j$, then $\mathcal{A} = \mathcal{A}'$, so that Theorems \ref{thm:GYsecondmoment} and \ref{thm:GYsecondmoment2} are consistent.  To see this, note that if $q=p^k = p^{2j}$, then $q \equiv 1 \pmod{4}$, and $(\frac{2 a_{\chi}}{q}) = 1$.
\begin{myprop}
\label{prop:recipe2}
Let conditions be as in Theorem \ref{thm:GYsecondmoment2}.
Then the modified recipe described in the introduction predicts the true asymptotic formula consistent with Theorem \ref{thm:GYsecondmoment2}.
\end{myprop}
\begin{proof}
We follow the proof of Proposition \ref{prop:recipe1}.  The diagonal term $\mathcal{D}$ is predicted in the same way as in Proposition \ref{prop:recipe1}, so we turn to the cross term \eqref{eq:epsiloncrossterm}.  Using Lemma \ref{lemma:gausssumsaverageQuadraticApprox} with $\ell_{\chi}$ replaced by $a_{\chi}$ (which is valid since the right hand side of \eqref{eq:gausssumsaverageQuadraticApprox} is periodic in $\ell_{\chi}$ modulo $p^{k-j}$)
evaluates \eqref{eq:epsiloncrossterm} as
\begin{multline}
\label{eq:epsilonCrossTermQuadraticCase}
X_{\alpha} 
\frac{\varphi(q_0)}{2}
\sum_{\pm}
\sum_{m,n} \frac{e_{p^k}(\pm mn)}{n^{1/2-\alpha} m^{1/2+\beta}} 
\Big(\frac{-2 a_{\chi}}{q}\Big)
\\
\epsilon_q e_{p^{k-2j}}\Big(\overline{2 a_{\chi}} \Big(\frac{a_{\chi} \pm mn}{p^j}\Big)^2 \Big)
  \delta(a_{\chi} \equiv \mp mn \mymod{p^j}).
\end{multline}
By a similar heuristic to that used in Proposition \ref{prop:recipe1}, the recipe suggests that we interpret the congruence $mn \equiv \mp a_{\chi} \pmod{p^j}$ by $mn = \mp b_{\chi}$ (assuming the sign is positive, since $m,n \geq 1$).  For either choice of sign, we have $e_{p^k}(\pm mn) = e_{p^k}(- b_{\chi}) = 1 + O(p^{j-k})$, so we replace this by $1$.  Additionally letting $\alpha =\beta= 0$, \eqref{eq:epsilonCrossTermQuadraticCase} is replaced by
\begin{equation*}
\frac{\varphi(q_0)}{2}
 \frac{d(|b_{\chi}|)}{\sqrt{|b_{\chi}|}} 
\Big(\frac{-2 a_{\chi}}{q}\Big)
\epsilon_q e_{p^{k}}(\overline{2 a_{\chi}} (a_{\chi} - b_{\chi})^2).
\end{equation*}
We also need to add the complex conjugate of this term to account for the other cross term.  
If $q \equiv 1 \pmod{4}$, then $\epsilon_q = 1$, $(\frac{-1}{q}) = 1$, and the two exponentials form a cosine, as given in \eqref{eq:A'def}.  If $q \equiv 3 \pmod{4}$, then $\epsilon_q =i$, $(\frac{-1}{q}) = -1$, and the combination of two terms forms a sine, as in \eqref{eq:A'def}.
\end{proof}
Remark.  The limitation to $j \geq k/3$ in Proposition \ref{prop:recipe2} occurs because there exist explicit evaluations of linear and quadratic exponential sums, but not cubic or higher.

\section{The upper bound and remarks on van der Corput shifting}
\label{section:upperbound}
\subsection{van der Corput shifting}
As a warm-up, we would like to re-interpret the classical van der Corput shifting method as an amplified second moment of a trigonometric polynomial.  
%Before doing that, we first describe amplification in general terms.  Suppose that $\lambda(n)$ is some sequence (say, supported on $1 \leq n \leq N$) and we would like to prove that it has some cancellation.  Let us hypothesize that $\lambda(n)$ fits into a family, that is, there exist some collection of sequences $(\lambda_f(n))$ for $f \in \mathcal{F}$, and $\lambda(n) = \lambda_{f_0}(n)$ for some particular $f_0 \in \mathcal{F}$.  As a baseline, suppose we have some kind of orthogonality relation that is able to produce a bound of the form
%\begin{equation}
%\sum_{f \in \mathcal{F}} \Big|\sum_n a_n \lambda_f(n) \Big|^2 \leq (|\mathcal{F}| + N) \sum_n |a_n|^2.
%\end{equation}
We begin by reviewing the classic method of van der Corput.
\begin{mylemma}[van der Corput shifting]
\label{lemma:vdC}
Let $a_n$ be any sequence of complex numbers supported on the interval $[1,N]$, and let $H \geq 1$ be an integer.  Then
\begin{equation}
\label{eq:vdC}
\Big| \sum_n a_n \Big|^2
\leq \Big(1 + \frac{N}{H}\Big) \sum_{|h| < H} \Big(1 - \frac{|h|}{H}\Big) \sum_{n} a_{n+h} \overline{a_n}.
\end{equation}
\end{mylemma}
Taking $a_n = e(f(n))$ gives the more standard formulation of the bound.  In practice, the goal is to reduce the oscillation of the exponential sum.  One can view this as a form of conductor-dropping at the archimedean place.

The standard proof of Lemma \ref{lemma:vdC} (as in say \cite[Section 2.3]{GK} or \cite[Section 8.3]{IK}) averages over the shift and uses Cauchy's inequality.  We will give an alternative proof which may be interpreted as an amplified second moment.
\begin{proof}
Let $b_m$ be a sequence of complex numbers supported on an interval containing at most $M$ integers.  
Let $S(x) = \sum_m b_m e(mx)$.
By Cauchy's inequality,
\begin{equation}
\label{eq:LinfinityL2}
|S(0)|^2 = 
\Big| \sum_m b_m \Big|^2 \leq M \sum_m |b_m|^2 = \int_0^{1} |S(x)|^2 dx.
\end{equation}
Define $D_H(x) = \sum_{1 \leq h \leq H} e(hx)$, which is the Dirichlet kernel.  Let $A(x) = \sum_n a_n e(nx)$.  
Note that $D_H(0) = H$, and that $S(x) := D_H(x) \cdot A(x)$ is supported on an interval containing at most $N+H$ integers.
Applying \eqref{eq:LinfinityL2} thus implies
\begin{equation}
\label{eq:LinfinityL2withDirichletKernel}
H^2 \Big| \sum_n a_n \Big|^2 
\leq (N+H) \int_0^{1} |D_H(x) A(x)|^2 dx.
\end{equation}
Next we calculate the integral on the right hand side of \eqref{eq:LinfinityL2withDirichletKernel}.  We have
\begin{equation}
\label{eq:amplifiedTrigL2calculation}
\int_0^{1} |D_H(x) A(x)|^2 dx
= 
\sum_{n_1, n_2} a_{n_1} \overline{a_{n_2}} 
\sum_{\substack{1 \leq h_1, h_2 \leq H \\ n_1 + h_1 = n_2 + h_2}} 1. 
\end{equation}
Write $h = h_1 - h_2$, so that $|h| \leq H$.  Note that the number of pairs $(h_1, h_2)$ such that $h_1 - h_2 = h$ equals $H-|h|$.
Then \eqref{eq:vdC} follows by direct simplification.
\end{proof}
Remarks.  The general inequality \eqref{eq:LinfinityL2} produces a pointwise bound from an $L^2$-norm.  By itself, this is unable to prove any cancellation, since for example it gives the same bound on $b_m \equiv 1$ as it does for $b_m = (-1)^m$.  However, we may view the variant 
\eqref{eq:LinfinityL2withDirichletKernel} as an amplified version of \eqref{eq:LinfinityL2}, where the idea is that $|D_H(x)|$ is a function concentrated near $x=0$.  Indeed, $D_H(0) = H$ yet $\int_0^{1} |D_H(x)| dx \ll \log{2H}$.  The mass of $|D_H(x)|$ is, roughly, concentrated in a $1/H$-neighborhood of $x=0$.  Thus, $D_H(x)$ amplifies the value of $A(x)$ near $x=0$.  Another motivating comment is that $D_H(x) A(x)$ is supported on an interval containing at most $N+H$ integers, which for $H \ll N$ is a benign increase compared to $N$.  Secondly, the diagonal (i.e., $h_1 = h_2$) contribution to \eqref{eq:amplifiedTrigL2calculation} is $H \cdot \sum_n |a_n|^2$ which grows linearly in $H$.  However, the left hand side of 
\eqref{eq:LinfinityL2withDirichletKernel} grows quadratically in $H$.  Hence the introduction of the Dirichlet kernel is able to overcome the barrier of the diagonal term (provided $H$ has some modest size).  This is one of the characteristic properties of an amplifier-- it saves a bit on the diagonal terms, but introduces additional off-diagonal terms.   

A potentially useful quality of this proof of Lemma \ref{lemma:vdC} is that it easily generalizes to multiple variables.

In \cite{HeathBrownHybrid}, Heath-Brown introduced a $q$-analog of the van der Corput shifting method.  We briefly summarize his method.  Suppose $\chi$ is a Dirichlet character of conductor $q$, and suppose $q_0 | q$.
Then one can perform a variation on the standard proof of Lemma \ref{lemma:vdC} but using shifts of the form $hq_0$, with $1 \leq h \leq H$.  In practice, this can be useful because $\chi(n+hq_0) \overline{\chi}(n) = \chi(1+h \overline{n} q_0)$ is periodic in $n$ (and $h$) of period $q/q_0$.  One can hence view this as a conductor-dropping technique, which can be beneficial for subsequent operations, such as Poisson summation.

An alternative way to set this method up is as follows.  Suppose that $a_n$ is some sequence supported on a finite interval.  Then by positivity,
\begin{equation*}
\Big|\sum_n a_n \chi(n) \Big|^2 \leq \sum_{\eta \in \chi H_{q_0}} \Big| \sum_n a_n \eta(n) \Big|^2.
\end{equation*}
On the other hand, by squaring out the right hand side and applying orthogonality of characters, we obtain
\begin{equation*}
 \sum_{\eta \in \chi H_{q_0}} \Big| \sum_n a_n \eta(n) \Big|^2
 = \varphi(q_0)
 \sum_{h \in \mz}
 \sum_{n} a_{n+h q_0} \overline{a_n} \chi(n+ h q_0) \overline{\chi}(n).
\end{equation*}
The diagonal contribution (from $h=0$) is $\varphi(q_0) \sum_{(n,q)=1} |a_n|^2$, which is acceptable to prove cancellation provided (say) $q_0 \ll N^{1-\delta}$ for some $\delta > 0$.  Hence we see the same phenomenon as in the classical van der Corput method where the diagonal term has been dampened at the expense of creating additional off-diagonal terms.

In this arrangement, if $a_n$ is supported on $[1,N]$ then $h$ is automatically restricted by $h \ll \frac{N}{q_0}$.  Heath-Brown's arrangement is a little more flexible in this regard since his method allows $h$ to be truncated at any point $H$ with $H \leq N/q_0$.  However, this is a rather minor technical point that can be addressed by also incorporating 
a $t$-aspect integral.  

Heath-Brown \cite[p.446]{HB12} observes that the usual form of van der Corput's inequality can be interpreted as a mean value, via \cite[Lemma 1.9]{Montgomery} (itself an inequality of Gallagher).  From this perspective, the above presentation is simply extending this method from the archimedean place to the $p$-adics.

\subsection{Hybrid second moment}
We conclude this section by sketching a proof of Theorem \ref{thm:HBmoment}. 
In \cite{GarciaYoung}, a similar bound to Theorem \ref{thm:HBmoment} was considered, but without the presence of the $t$-integral.  A reader who seeks more details than are provided in this sketch can consult \cite{GarciaYoung} for a more rigorous treatment of the problem.
Our main goal with this sketch is to make it apparent why it is reasonable to attribute it to Heath-Brown \cite{HeathBrownHybrid}.  

Let $\omega(t)$ be smooth, nonnegative, supported on $[-1,2]$, and satisfying $\omega(t) =1$ on $[0,1]$.
By standard reduction steps (approximate functional equation, dyadic partition of unity, etc.)
it suffices to estimate
\begin{equation*}
I(T,T_0, q, q_0, M) = 
\intR \omega\Big(\frac{t-T}{T_0}\Big)
\sum_{\psi \in H_{q_0}}
\Big|\sum_{n} \frac{\chi(n) \psi(n)}{n^{1/2+it}} W_M(n) \Big|^2 dt,
\end{equation*}
where $W_M(x)$ is a smooth function supported on $M/2 \leq x \leq M$, and satisfying derivative bounds $\frac{d^j}{dx^j} W_M(x) \ll_j M^{-j}$, for each $j=0,1,2, \dots$.  Indeed, we have that the second moment is
$\ll (qT)^{\varepsilon} \max_{1 \leq M \ll (qT)^{1/2+\varepsilon}} I(T, T_0, q, q_0, M)$.  By squaring this out and performing the sum over $\psi$ and integral over $t$, we obtain
\begin{equation*}
I(T,T_0, q, q_0, M) =
T_0 \varphi(q_0) \sum_{\substack{n_1, n_2 \\ n_1 \equiv n_2 \shortmod{q_0}}} \frac{\chi(n_1) \overline{\chi(n_2)}}{ n_1^{1/2+iT} n_2^{1/2-iT}} 
\widehat{\omega}\Big(\frac{T_0}{2 \pi} \log\frac{n_1}{n_2}\Big)
W_M(n_1) \overline{W_M(n_2)}
.
\end{equation*}
Now write $n_1 = n_2 + hq_0$ and rename $n_2$ to $n$.  Then
\begin{equation*}
I(T,T_0, q, q_0, M) =
T_0 \varphi(q_0)
\sum_{h \in \mz} J(h),
\end{equation*}
where $J(h) = J(h,\chi, q_0, T, T_0, M)$ is defined by
\begin{equation*}
J(h) = \sum_n \frac{\chi(n+hq_0) \overline{\chi(n)}}{ (n+hq_0)^{1/2+iT} n^{1/2-iT}} 
\widehat{\omega}\Big(\frac{T_0}{2 \pi} \log\big(1+\frac{hq_0}{n}\big)\Big)
W_M(n+hq_0) \overline{W_M(n)}.
\end{equation*}

Let us dispense with the diagonal term $h=0$.  We have $|J(0)| \ll 1$, and this gives a bound consistent with Theorem \ref{thm:HBmoment}.  To complete the proof, we will show the following.
\begin{myprop}
\label{prop:Jbound}
For $1 \ll M \ll (qT)^{1/2+\varepsilon}$, we have
\begin{equation*}
 \sum_{h \neq 0} |J(h)| \ll (qT)^{\varepsilon} \Big(\frac{q^{1/2} T^{1/2}}{q_0^{3/2} T_0^{3/2}} + 1\Big).
\end{equation*}
\end{myprop}
Proposition \ref{prop:Jbound} directly implies Theorem \ref{thm:HBmoment}.
Before embarking on its proof, we state the following crucial lemma due to Heath-Brown.
For $q$ odd and $q_0 | q$, with $\chi$ of conductor $q$, define the following character sum:
\begin{equation*}
S(\chi, h q_0,n) = 
\sum_{\alpha \shortmod{q}} \chi(\alpha + h q_0) \overline{\chi}(\alpha) e_{q}(\alpha n).
\end{equation*}
\begin{mylemma}[Heath-Brown, \cite{HeathBrownHybrid}, Lemma 9]
\label{lemma:HBlemma}
We have
\begin{equation*}
\sum_{1 \leq |h| \leq A} \sum_{1 \leq |n| \leq B} |S(\chi, h q_0, n)| \ll q^{1/2+\varepsilon} (AB q_0^{-1/2}  + (qq_0 A)^{1/4}),
\end{equation*}
and
\begin{equation*}
\sum_{1 \leq |h| \leq A} |S(\chi, h q_0, 0)| \ll  q_0 A q^{\varepsilon}.
\end{equation*}
\end{mylemma}
Remark.  Heath-Brown technically estimated a slightly different sum but by some symmetry arguments and using that $q$ is odd, his work easily implies the version as stated above.

\begin{proof}[Proof of Proposition \ref{prop:Jbound}]
By Poisson summation (with modulus $q$), we have
\begin{equation}
\label{eq:Jhformula}
J(h) = 
\frac{1}{q} \sum_{n \in \mz} S(\chi, h  q_0, n) \cdot K(h, n, q_0, T, T_0, M),
\end{equation}
where 
\begin{equation*}
K(h, n,q_0, T, T_0, M) = \intR \frac{\widehat{\omega}\Big(\frac{T_0}{2 \pi} \log\big(1+\frac{hq_0}{t}\big)\Big)}{(t+hq_0)^{1/2+iT} t^{1/2-iT}} e\Big(\frac{-nt}{q}\Big) W_M(t+hq_0) \overline{W_M(t)} dt.
\end{equation*}
First we attend to the integral $K$.  The support of $W_M$ means that we may assume $t \asymp M$.  The rapid decay of $\widehat{\omega}$ means that we may assume
\begin{equation}
\label{eq:hbound}
h \ll \frac{M}{q_0 T_0} (qT)^{\varepsilon}.
\end{equation}
We can view $K$ as an integral of the form
\begin{equation*}
\frac{1}{M} \intR e^{-i \phi(t)} g(t) dt,
\end{equation*}
where $g$ is a nice bump function supported on $[M, 2M]$ and 
$\phi(t) = T  \log(1+\frac{h q_0}{t}) + \frac{2\pi nt}{q}$.  Note that
\begin{equation*}
\phi'(t) = T \frac{-hq_0}{t(t+hq_0)} + \frac{2\pi n}{q},
\end{equation*}
and
\begin{equation*}
\phi''(t) = T \frac{( 2t + h q_0) hq_0}{t^2 (t+hq_0)^2}.
\end{equation*}
An integration by parts argument shows that $K$ is very small except if
\begin{equation*}
n \ll (qT)^{\varepsilon} \Big( \frac{Tq |h|   q_0}{M^2} + \frac{q }{M} \Big) =: N'.
\end{equation*}
With this truncation in hand, we may then bound $K$ using stationary phase (e.g. see \cite[Prop 8.2]{BKY}), since $|\phi''(t)| \asymp \frac{T  |h| q_0}{M^3} \neq 0$ on the support of $W_M(t+hq_0) \overline{W_M(t)}$.  This method gives the bound
\begin{equation*}
K(h, n, q_0, T, T_0, M) \ll \frac{M^{1/2}}{(T q_0 |h|)^{1/2}}.
\end{equation*}

Applying this information into \eqref{eq:Jhformula} implies
\begin{equation*}
|J(h)| \ll \frac{1}{q} \frac{M^{1/2}}{(T q_0 |h|)^{1/2}} \sum_{n \ll N'} |S(\chi, h q_0, n)| +(qT)^{-100}.
\end{equation*}
Hence
\begin{equation}
\label{eq:JboundNearEnd}
\sum_{|h| \asymp H} |J(h)|
\ll 
\frac{1}{q} \frac{M^{1/2}}{(T q_0 H)^{1/2}} 
\Big(
\sum_{|h| \asymp H}
\sum_{1 \leq |n| \ll N'} |S(\chi, h q_0, n)|
+ 
\sum_{|h| \asymp H} |S(\chi, hq_0, 0)| \Big),
\end{equation}  
plus a negligible error.  Substituting Lemma \ref{lemma:HBlemma}, the contribution to the right hand side of \eqref{eq:JboundNearEnd} from $n \neq 0$ is bounded by
\begin{equation*}
\frac{(qT)^{\varepsilon}}{q} \frac{M^{1/2}}{(T q_0 H)^{1/2}} \Big( q^{1/2} \frac{H}{q_0^{1/2}} \Big(\frac{T H q q_0}{M^2} + \frac{q}{M}\Big)  + (q q_0 H)^{1/4} \Big).
\end{equation*}
Using $1 \ll H \ll \frac{M}{q_0 T_0} (qT)^{\varepsilon}$ and simplifying
gives a bound consistent with Proposition \ref{prop:Jbound}.  Similarly, the contribution from $n=0$ is bounded by
\begin{equation*}
\frac{(qT)^{\varepsilon}}{q} \frac{M^{1/2}}{(T q_0 H)^{1/2}} q_0 H
\ll  \frac{M (qT)^{\varepsilon}}{q T^{1/2} T_0^{1/2}} \ll \frac{(qT)^{\varepsilon}}{(q_0 T_0)^{1/2}},
\end{equation*}
which is better than claimed in Proposition \ref{prop:Jbound}.
\end{proof}

\bibliographystyle{amsalpha}
\bibliography{Gauss}	
\end{document}